\documentclass[11pt,a4paper]{article}
\usepackage{times}
\usepackage{graphicx}
\usepackage{amsfonts}
\usepackage{mathtools} 
\usepackage{esint} 
\usepackage{amsthm}
\usepackage{tikz}
\usetikzlibrary{shapes,arrows,positioning}
\usepackage{amssymb,epsfig,epstopdf,titling,url,array}
\usepackage{textcomp}
\usepackage{multirow}
\usepackage[colorlinks=true,linkcolor=blue, citecolor=blue]{hyperref}%
\usepackage[square,numbers]{natbib}

\theoremstyle{plain}
\newtheorem{thm}{Theorem}[section]

\newtheorem{prop}[thm]{Proposition}

\theoremstyle{definition}
\newtheorem{defn}{Definition}[section]

\theoremstyle{remark}

\makeatletter
\newbox\xrat@below
\newbox\xrat@above
\newcommand{\xrightarrowtail}[2][]{%
  \setbox\xrat@below=\hbox{\ensuremath{\scriptstyle #1}}%
  \setbox\xrat@above=\hbox{\ensuremath{\scriptstyle #2}}%
  \pgfmathsetlengthmacro{\xrat@len}{max(\wd\xrat@below,\wd\xrat@above)+.6em}%
  \mathrel{\tikz [>->,baseline=-.75ex]
                 \draw (0,0) -- node[below=-2pt] {\box\xrat@below}
                                node[above=-2pt] {\box\xrat@above}
                       (\xrat@len,0) ;}}
\makeatother

\begin{document}
\pagestyle{empty}


\begin{center}
{\fontsize{14}{20}\bf
The biquotient model of the total space of the projectivization quaternionic bundles over certain manifolds.
}
\end{center}

\begin{center}
 \textbf{Meshach Ndlovu }\\
Department of Mathematics and Statistics,\\ Faculty of Computational Sciences,\\ Gwanda State University,\\ e-mail: meshach.ndlovu@gsu.ac.zw
\end{center}

\section*{Abstract}

In this paper, we compute the rational homotopy type of the quaternionic projective bundle $P(\tau): \mathbb{H}P^{n-1} \rightarrow P(E) \rightarrow M$ obtain from the quaternionic tangent bundle $\tau: \mathbb{H}^{n} \rightarrow E \rightarrow M$  over quaternionic manifold $M$. If the base space $M$ is the quaternionic projective space $\mathbb{H}P^{2}$, 
then the rational homotopy type of the total space $P(E)$ of the quaternionic projective bundle $P(\tau)$  is given by a biquotient model $Sp(1)\backslash Sp(3) / Sp(1) \times Sp(1)$. 

\vspace{3mm}
\textbf{Keywords:} Quaternionic Manifold; Projectivization Bundle; Sullivan Models; Biquotient Model.



\section{Introduction}
Let \( M \) be a smooth manifold of dimension \( 4n \). If \( M \) admits a torsion-free \( G \)-structure, where \( G \) denotes the maximal subgroup of \( \mathrm{GL}(n, \mathbb{H}) \times \mathrm{GL}(1, \mathbb{H}) \in \mathrm{GL}(4n, \mathbb{R}) \), then \( M \) is referred to as a quaternionic manifold. Under the natural inclusion of Lie groups
\( \iota: \mathrm{GL}(n, \mathbb{H}) \hookrightarrow \mathrm{GL}(4n, \mathbb{R}), \)
the structure group of the tangent bundle of \( M \) reduces to \( \mathrm{GL}(n, \mathbb{H}) \times \mathrm{Sp}(1) \), where \( \mathrm{Sp}(1) \) is a symplectic group, which plays a role analogous to \( \mathrm{U}(1) \) in complex geometry.
 Also, $M$ is a quaternionic Kähler manifold if $M$ is a quaternionic manifold with a torsion-free $Sp(n)\times Sp(1)$-structure where $Sp(n)$ is a symplectic group of dimension $n$. There are many examples of quaternionic manifolds and these include 4-dimensional spheres $S^{4}$, quaternionic projective space $\mathbb{H}P^n$, quaternionic Grassmannian manifold $\mathrm{G}_{n,k}(\mathbb{H})$, Quaternionic Calabi-Yau Manifolds and Hyperkähler manifolds \cite{salamon1986differential}. Moreover, a quotient space $Sp(n+1)/Sp(n) \times Sp(1)$ can be used to describe the quaternion projective space $\mathbb{H}P^n$ \cite{kraines1966topology}.
If \( M \) is a symplectic manifold with symplectic form \( \omega \) with tangent bundle \( TM \). Then the Pontryagin classes of the tangent bundle are given by the following formal power series in the curvature form \( \Omega \) of a connection on \( TM \)

\[
p(T(M)) = 1 + p_1(TM) + p_2(TM) + \cdots,
\]

where the Pontryagin class \( p_i(TM) \) is given as

\[
p_i(TM) = (-1)^i \frac{1}{(2\pi)^{2i}} \text{tr} \left( \Omega^{2i} \right).
\]

For a symplectic manifold, one can consider specific connections compatible with the symplectic structure [for reference see \cite[\textsection~III]{borel1958characteristic}, p. 488].

Given a quaternionic tangent bundle $\tau: \mathbb{H}^{n} \rightarrow E \rightarrow M$  over quaternionic manifold $M$. The projectivization bundle $P(\tau): \mathbb{H}P^{n-1} \rightarrow P(E) \rightarrow M$ is obtained by taking the projectivization of each fiber in $\tau$, which is replacing $\mathbb{H}^n$ with $\mathbb{H}P^{n-1}$. Also, the projective general linear group $\text{PGL}(n, \mathbb{H})$ is simply the general linear group $\text{GL}(n, \mathbb{H})$ modulo scalar matrices \cite[\textsection~2]{mimura1970cohomology}. Then the structure group of the projectivization bundle can be understood in terms of its transition functions $\overline{g}_{\alpha \beta} : \mathrm{U}_{\alpha} \bigcap \mathrm{U}_{\beta} \rightarrow \text{PGL}(n, \mathbb{H})$ induced from $g_{\alpha \beta} : \mathrm{U}_{\alpha} \bigcap \mathrm{U}_{\beta} \rightarrow \text{GL}(n, \mathbb{H})$. As vector spaces the cohomology algebra of the total space of the projectivization bundel $P(\tau)$ can be obtained as $H^*(P(E), \mathbb{Q}) = H^*(M, \mathbb{Q}) \otimes H^*(\mathbb{H}P^{n-1}, \mathbb{Q})$ by applying Leray-Hirsch theorem \cite[p. 50]{Bott1982}. The $i^{th}$ Pontryagin classes of $\tau$, denoted as $p_i(\tau)$, are cohomology classes in the cohomology ring $H^{4i}(M, \mathbb{Q})$. Thus, the cohomology algebra of $P(E)$ is given by \[ H^*\left(P(E), \mathbb{Q}\right) = H^*(M, \mathbb{Q})[x_4]/\left( \sum_{i=0}^{n} p_i(\tau)x^{n-i}_4 \right), \] where $p_i(\tau)$ are the Pontryagin classes with $p_0(\tau) = 1$ and $x_4 \in H^4\left(\mathbb{H}P^{n-1}, \mathbb{Q}\right)$. Recall that the projectivization of quaternionic vector bundles is analogous to the projectivization of complex vector bundles \cite[\textsection~20]{Bott1982}.

In this paper, we extend our previous results (See \cite{gastinzi2024rational} and \cite{ndlovu2025note}) on the projectivization of complex vector bundles to the setting of quaternionic vector bundles. In particular, we explore the rational homotopy type of the total space $P(E)$ of quaternionic projective bundles and establish the following theorems:

\textbf{Theorem \ref{thm33}} \textit{ The projectivization of a quaternionic vector bundle over a $4n$-dimensional sphere $ \tau: \mathbb{H}^{n} \rightarrow E \rightarrow S^{4n}$ is given by $$P(\tau) : \mathbb{H}P^{n-1} \rightarrow P(E) \rightarrow S^{4n}.$$ Then the total space $P(E)$ of the projectivization bundle $P(\tau)$ has a rational homotopy type of $Sp(1) \setminus Sp(2n) / Sp(2n-1)$.}

\textbf{Theorem \ref{thm34}} \textit{The total space $P(E)$ of the projectivization bundle  $$P(\tau) : \mathbb{H}P^{1} \rightarrow P(E) \rightarrow \mathbb{H}P^{2}$$ has a rational homotopy type that is quasi-isomorphic to biquotient space $$Sp(1)\setminus Sp(3) / Sp(1) \times Sp(1).$$}

This paper is organized as follows: In \textsection 2, we offer key definitions in rational homotopy theory together with Sullivan models. In \textsection 3, we provide our key findings.

\section{Sullivan Models}
The primary reference for the definitions in this paper is \cite{Felix2001}.

\begin{defn} \cite[\textsection~3]{Felix2001} Let $V = \oplus_{p\geq 0}V^{p}$ be a graded vector space. A graded algebra is a graded vector space $A = \oplus_{p\geq 0}A^{p}$ together with a degree zero associative multiplication
    \[ A^{p} \otimes A^{q} \rightarrow A^{p+q}, \;\;\;\;\; x \otimes y \rightarrow xy, \]
\end{defn}
and $1 \in A^{0}$. This graded algebra $A$ is called a commutative graded algebra (cga) if $xy = (-1)^{\vert x \vert \vert y \vert}yx,$ where the homogeneous elements $x$ and $y$  have degree $\vert x \vert$ and $\vert y \vert$ respectively.

\begin{defn} \cite[\textsection~3]{Felix2001}
    A differential graded algebra (dga) is a pair $(A, d)$, where the differential maps are given by  $d: A^{p} \rightarrow A^{p+1}$ and  $d(xy) = (dx)y+(-1)^{\vert x \vert}x(dy)$. If $A$ is commutative, then $(A, d)$ is called a commutative differential graded algebra (cdga). 
\end{defn}

\begin{defn} \cite[\textsection~3]{Felix2001}
    The free commutative graded algebra over a graded vector space $V = \oplus_{p\geq 0}V^{p}$, is $\Lambda V = S(V^{even}) \otimes E(V^{odd})$ defined as a product of symmetric algebra $V = S(V^{even})$ and exterior algebra $E(V^{odd})$. Also, $\Lambda V$ is generated by the elements $v_i \in V$.
\end{defn}

\begin{defn}
\cite[\textsection~3]{Felix2001}  The graded vector space defined by $(sV)^{n} = V^{n+1}$ is the suspension $sV$ of the graded vector space $V$.
\end{defn}

\begin{defn} Sullivan defined a cdga $A_{PL}(X, \mathbb{Q})$ of a topological space $X$ together with natural cochain algebra quasi-isomorphisms such that
    $$C^{*}(X, \mathbb{Q}) \xrightarrow{\cong} D(X, \mathbb{Q}) \xleftarrow{\cong} A_{PL}(X, \mathbb{Q}),$$  where $D(X, \mathbb{Q})$ is a natural cochain algebra \cite[\textsection 10]{Felix2001}. 
\end{defn}

\begin{defn} \cite[\textsection~12]{Felix2001}
The Sullivan algebra $(\Lambda V , d)$ is a free commutative graded algebra together with a differential $d$ 
such that, $dV = 0 \; \text{in} \; V(0) \; \; \text{and} \; \; d : V(k) = \Lambda V(k-1), \;\; k \geq 1.$  The pair $(\Lambda V, d)$ is called the Sullivan minimal model of a simply connected topological space $X$ if there exists a quasi-isomorphism $ m : (\Lambda V, d) \rightarrow (A, d)$, where $(A, d)$ is the cdga of $X$ and $ d \subset \Lambda ^{\geq 2}V$.
\end{defn}


\begin{defn} \cite[\textsection~12]{Felix2001}
If there exists a homomorphism $\phi : (\Lambda V, d) \rightarrow H^{*}(\Lambda V, d)$ that induces an isomorphism on cohomology, then a Sullivan minimal model $(\Lambda V, d)$ is said to be formal. If the minimal model of a space $X$ is formal, then $X$ is also formal.
\end{defn}

\begin{defn} \cite[\textsection~14]{Felix2001}
   The relative Sullivan algebra is defined by $(A \otimes \Lambda V, d)$, where $(A, d)$ is a Sullivan algebra. Moreover, $V = \bigoplus_{k\geq 1}V^k$ is a graded vector space admitting an increasing sequence $V(0) \subset V(1) \subset \cdots V(k)$ of graded subspaces such that $d : V(0) \rightarrow A  \; \; \text{and} \; \; d : V(k) = A \otimes \Lambda V(k-1), \;\; k \geq 1.$ 
\end{defn}


\begin{defn} \cite[\textsection~32]{Felix2008}
    A pure Sullivan algebra $(\Lambda V, d) = (\Lambda Q \otimes \Lambda P, d)$ is a cdga where the graded vector space $V$ decomposes to $Q$ concentrated in even degrees and $P$ concentrated in odd degrees and the differential satisfies $dQ = 0$ and $dP \subseteq \Lambda Q$.
\end{defn}

\begin{defn} \cite[\textsection~3]{Felix2008}
Suppose $K$ acts freely on $G/H$ and that $H$ and $K$ are closed connected subgroups of a compact connected Lie group $G$. The quotient, $K\setminus G / H$, is a closed manifold with $KxH$, abnd $x \in G$. A biquotient of $G$ is any closed manifold diffeomorphic to $K\setminus G / H$.
\end{defn}

\begin{thm} \cite[\textsection~3]{Felix2008}
    Let $H$ and $K$ be closed connected subgroups of a
compact connected Lie group $G$ defines a biquotient $K\setminus G\; /\; H$. We denote
by $i_{H} : H \rightarrow G$, $i_{K} : K \rightarrow G $ the canonical inclusions and by $B_{i_{H}} : BH \rightarrow
BG,\;\; B_{i_{K}} : BK \rightarrow BG$ the induced maps on classifying spaces. Let $\Lambda V =
H^{*}(BG; \mathbb{Q})$, $ \Lambda W_{H} = H^{*}(BH; \mathbb{Q})$, $\Lambda W_{K} = H^{*}(BK; \mathbb{Q})$ be the cohomology algebras of $BG$, $BH$ and $BK$ respectively. We denote by $sV$ a copy of the
vector space $V$ shifted by one degree, $\vert sv \vert = \vert v \vert - 1$ if $v \in V$, and define a
differential $d$ on $\Lambda W_{H} \otimes \Lambda W_{K} \otimes \Lambda(sV)$ by $dw = 0$ if $w \in W_{H} \oplus W_{K}$ and
$d(sv) = H^{*}(B_{i_{H}})(v) - H^{*}(B_{i_{K}})(v)$ if $sv \in sV$.
Then the cdga $(\Lambda(W_{H} \oplus W_{K}) \otimes \Lambda(sV), d)$ is a model for the biquotient
$K\setminus G / H$. In particular, $H^{*}(K\setminus G / H ; \mathbb{Q}) = H^{*}(\Lambda(W_{H} \oplus W_{K}) \otimes \Lambda(sV), d)$. \label{thm21}
\end{thm}

Theorem \ref{thm21} provides a Sullivan model $(\Lambda(W_{H} \oplus W_{K}) \otimes \Lambda(sV), d)$ for the biquotient space $K \setminus G / H$ in terms of the classifying space maps $B_{i_H}$ and  $B_{i_K}$ induced by the subgroup inclusions $i_H$ and  $i_K$ respectively.

\begin{prop}
\label{prop22}     Let $M$ be a quaternion manifold, then the quaternion fibre bundle given by $\tau : \mathbb{H}^{n} \rightarrow E \rightarrow M$ have its projectivization bundle as $P(\tau) : \mathbb{H}P^{n-1} \rightarrow P(E) \rightarrow M$.  Suppose $(A, d)$ is the model of $M$, then the model of the total space $P(E)$ is  $$(A \otimes \Lambda(x_{4}, x_{4n-1}), D),$$ where $\displaystyle Dx_{4n-1}=x_{4}^{n}+\sum_{i=1}^{n}p_{i}x_{4}^{n-i}$, $p_{i} \in H^{4i}(A, \mathbb{Q})$ denotes the Pontryagin classes and $ D_{|A}$.

\end{prop}

\begin{proof}[\textbf{Proof:}]
 Let $M$ be a quaternionic manifold and $\tau$ to the tangent bundle. The structure group of $M$ is $Sp(n)$ and the morphism $f: M \rightarrow BSp(n)$, is the classifying map 
 of $P(\tau)$. Moreover, $\Lambda(y_{4}, \cdots y_{4n})$ is the model of $BSp(n)$. Let $(A, d)$ be a cdga model of $M$. Then  $f$ can be modelled by  $$\phi : (\Lambda(y_{4}, \cdots y_{4n}), 0) \rightarrow (A, d),$$ where $[p_{i}]\in H^{2i}(A, d)$ for $i=1,\cdots , n$ are the Pontryagin classes. The relative Sullivan model of the projectivization bundle $P(\tau)$ is 
 \[\Phi : (A, d) \rightarrow (A \otimes (\Lambda(x_{4}, x_{4n-1}), D) \rightarrow (\Lambda(x_{4}, x_{4n-1}), d), \] where \[Dx_{4n-1}=x_{4}^{n}+\sum_{i=1}^{n}p_{i}x_{4}^{n-i} \;\;\; \text{and}\;\;\; D_{|A}.\] 

\end{proof}

\section{Main Results}
This section establishes foundational results in the rational homotopy theory of projectivized quaternionic bundles $P(\tau) : \mathbb{H}P^{n-1} \to P(E) \to M$ and biquotient spaces $K \setminus G / H$. We begin by computing the Sullivan models of $Sp(1) \setminus \left( Sp(1) \times Sp(n-1) \right) / Sp(n-1)$ and $Sp(1) \setminus \ Sp(n+1) / Sp(1) \times Sp(n-1)$. Finally , we prove the following rational homotopy equivalences: $P(E) \simeq_{\mathbb{Q}} Sp(1) \setminus Sp(2n) / Sp(2n-1)$ for $M=S^{4n}$ and $P(E) \simeq_{\mathbb{Q}} Sp(1) \setminus Sp(3) / Sp(1)\times Sp(1)$ for $M=\mathbb{H}P^2$.

\begin{prop}
  For \( n \geq 2 \), the biquotient space
\[
Sp(1) \setminus \left( Sp(1) \times Sp(n-1) \right) / Sp(n-1)
\]
admits a contractible Sullivan model.
\end{prop}

\begin{proof}
Here we compute the Sullivan model of the biquotient space
$Sp(1)\setminus Sp(1) \times Sp(n-1) / Sp(n-1).$ This space is defined via the diagram
\[
Sp(1)\xrightarrow{i_K} Sp(1) \times Sp(n-1) \xleftarrow{i_H} Sp(n-1),
\]
where $i_K$ and $i_H$ denote the canonical inclusions. We use the following cdga models for the classifying spaces:
\begin{align*}
   & BSp(1) = (\Lambda a_4, 0),
    \; BSp(n-1) = (\Lambda(v_4, v_8, \dots, v_{4n-4}), 0),
     \\ & \text{and} \; B(Sp(1)\times Sp(n-1)) = (\Lambda b_4 \otimes \Lambda(z_4, z_8, \dots, z_{4n-4}), 0).
\end{align*}
The classifying maps yield morphisms
\[
(\Lambda(v_4, \dots, v_{4n-4}), 0) \xleftarrow{\varphi_H} (\Lambda b_4 \otimes \Lambda(z_4, \dots, z_{4n-4}), 0) \xrightarrow{\varphi_K} (\Lambda a_4, 0),
\]
defined on generators by:
\begin{align*}
    &\varphi_K(b_4) = \varphi_K(z_4) = a_4, \quad \varphi_K(z_{4i}) = 0 \text{ for } 2 \leq i \leq n-1, \\
    &\varphi_H(b_4) = v_4, \quad \varphi_H(z_{4i}) = v_{4i} \text{ for all } i.
\end{align*}
The Sullivan model of the biquotient space $Sp(1)\setminus Sp(1) \times Sp(n-1) / Sp(n-1)$ is given by
\[
(\Lambda(a_4, v_4, \dots, v_{4n-4}, b_3, z_3, \dots, z_{4n-5}), d),
\]
with 
\begin{align*}
    & da_4 = dv_4= dv_{4i} = 0 ,  db_3 =   dz_3 = v_4 - a_4,  \\
    & dz_7 = v_8 - a_4^2,  dv_{11} = v_4 v_8 + v_{12} - a_4^3, \cdots ,\\
    & dv_{4n-9} = v_4 v_{4n-12} + v_8 v_{4n-16} + \cdots + v_{4n-8} - a_4^{n-2}, \\
    & dv_{4n-5} = v_4 v_{4n-8} + v_8 v_{4n-12} + \cdots + v_{4n-4} - a_4^{n-1}.
\end{align*}
Introducing a new variable $t_4 = v_4 - a_4$. The Sullivan model becomes
\[
(\Lambda(t_4, v_8, \dots, v_{4n-4}, b_3, z_3, \dots, z_{4n-5}), d),
\]
with 
\begin{align*}
    & db_3 = dz_3 = t_4, dz_7 = v_8 - a_4^2, dv_{11} = (t_4 + a_4) v_8 + v_{12} - a_4^3, \cdots, \\
    & dv_{4n-9} = (t_4 + a_4) v_{4n-12} + v_8 v_{4n-16} + \cdots + v_{4n-8} - a_4^{n-2}, \\
    & dv_{4n-5} = (t_4 + a_4) v_{4n-8} + v_8 v_{4n-12} + \cdots + v_{4n-4} - a_4^{n-1}.
\end{align*}
We cancel the acyclic ideals generated by $(b_3, z_3, t_4)$, obtaining the reduced model
\[
(\Lambda(v_8, \dots, v_{4n-4}, z_7, \dots, z_{4n-5}), d),
\]
where
\begin{align*}
    & dv_8 = \cdots = dv_{4n-8} = dv_{4n-4} =0,\; dz_7 = v_8,\;  dv_{11} = v_{12}, \cdots, \\
    & dv_{4n-9} = v_8 v_{4n-16} + \cdots + v_{4n-8},  dv_{4n-5} = v_8 v_{4n-12} + \cdots + v_{4n-4}.
\end{align*}

Now, note the existence of linear differentials such as $dz_7 = v_8$ and $dz_{11} = v_{12}$. Canceling these acyclic ideals generated by $(z_7, v_8)$ and $(z_{11}, v_{12})$,  we obtain the Sullivan model
\[
(\Lambda(v_{16}, \dots, v_{4n-4}, z_{15}, \dots, z_{4n-5}), d),
\]
with
\begin{align*}
    & dv_{16}=\cdots dv_{4n-4} = 0, dz_{15} = v_{16}, \cdots, \\
    & dv_{4n-9} = v_{16} v_{4n-24} + \cdots + v_{4n-8}, \\
    & dv_{4n-5} = v_{16} v_{4n-20} + \cdots + v_{4n-4}.
\end{align*}

Continuing this cancellation process up to $t_{4n-4} = v_{4n-4}$ and $dz_{4n-5} = t_{4n-4}$, we get the Sullivan model as
\[
(\Lambda(t_{4n-4}, v_{4n-5}), d), \; 
\]
with \[dt_{4n-4} = 0 \; \text{and}  \; dv_{4n-5} = t_{4n-4}.\] Hence, the cdga model of  $Sp(1)\setminus Sp(1)\times Sp(n-1)/Sp(n-1)$ is contractible and not minimal.
\end{proof}

\begin{prop}
Let \( n \geq 2 \), the biquotient space
\(Sp(1) \setminus Sp(n+1) / Sp(1) \times Sp(n-1)
\)
admits a Sullivan model given by 
\[
\left( \Lambda(c_4, b_4, a_{4n-1}, a_{4n+3}), d \right),
\]
where 
\[
dc_4 = db_4 = 0, \quad 
da_{4n-1} = (-b_4 + \beta_3 c_4) \cdot b_{4n-4} - \beta_{4n-1} c_4^n, \quad 
da_{4n+3} = -\beta_{4n+3} c_4^{n+1},
\]
for certain rational coefficients \( \beta_r \in \mathbb{Q} \), with \( r = \{3, 7, \dots 4n+3\}\). The polynomial \( b_{4n-4} \) is recursively defined by
\[
\begin{aligned}
b_{4n-4} &= (b_4 - \beta_3 c_4) \Big[ (b_4 - \beta_3 c_4) \Big[ \cdots \Big[ (b_4 - \beta_3 c_4) \big( (b_4 - \beta_3 c_4) b_4 + \beta_7 c_4^2 \big) \\
&\quad + \beta_{11} c_4^3 \Big] \cdots + \beta_{4n-13} c_4^{n-3} \Big] + \beta_{4n-9} c_4^{n-2} \Big] + \beta_{4n-5} c_4^{n-1}.
\end{aligned}
\]
\end{prop}

\begin{proof}
Let us compute the Sullivan model of the biquotient space
\[
Sp(1) \setminus Sp(n+1) / Sp(1) \times Sp(n-1),
\]
with the following inclusions maps
\[
Sp(1) \xrightarrow{i_K} Sp(n+1) \xleftarrow{i_H} Sp(1) \times Sp(n-1).
\]

The cdga models the classifying spaces are as follows:
\begin{align*}
    & BSp(n+1) = (\Lambda(a_4, a_8, \dots, a_{4n+4}), 0), BSp(1)=(\Lambda y_4, 0) \\
    & BSp(1) \times BSp(n-1) = (\Lambda(b_4, b_8, \dots, b_{4n-4}) \otimes \Lambda(x_4), 0).
\end{align*}

We now consider the pullback diagram of classifying spaces
\[
(\Lambda(b_4, \dots, b_{4n-4}) \otimes \Lambda(x_4), 0) 
\xleftarrow{\varphi_H} 
(\Lambda(a_4, \dots, a_{4n+4}), 0)
\xrightarrow{\varphi_K}
(\Lambda(c_4), 0),
\]
where the maps \(\varphi_H\) and \(\varphi_K\) are induced by the classifying maps $i_H$ and $i_K$ respectively.  To construct the Sullivan model of $Sp(1) \setminus Sp(n+1) / Sp(1) \times Sp(n-1)$, we make use of Theorem \ref{thm21} to obtain
\[
(\Lambda(x_4, c_4, b_4, \dots, b_{4n-4},  a_3, a_7, \dots,a_{4n-5}, a_{4n-1}, a_{4n+3}), d),
\]
with 
\begin{align*}
    & dx_4=dc_4=db_4=db_8=\cdots=db_{4n-4}=0,  da_3 =  x_4+b_4-\beta_{3}c_4, \\
    & da_7 = x_4b_4+b_8 - \beta_{7}c_4^2, da_{11} = x_4b_8+b_{12} - \beta_{11}c_4^3, \cdots, \\
    & da_{4n-5}=x_4b_{4n-8} + b_{4n-4} - \beta_{4n-5}c_4^{n-1}, da_{4n-1}=x_4b_{4n-4} - \beta_{4n-1}c_4^{n-1}, \\
    & da_{4n+3}= - \beta_{4n+3}c_4^{n-1} \;\; \text{for}\;\; r=\{3,7,\cdots, 4n+3\} \;\; \text{such that}\;\; \beta_{r} \in \mathbb{Q}.
\end{align*}
Since, $da_3 =  x_4+b_4-\beta_{3}c_4 $ is a linear differential. We apply the change of variable $x_4 =t_4-b_4+\beta_{3}c_4$ to obtain the following Sullivan model  
\[
(\Lambda(t_4, c_4, b_4, \dots, b_{4n-4},  a_3, a_7, \dots,a_{4n-5}, a_{4n-1}, a_{4n+3}), d),
\]
with 
\begin{align*}
    & dt_4=dc_4=db_4=db_8=\cdots=db_{4n-4}=0,  da_3 =  t_4, \\
    & da_7 = (t_4-b_4+\beta_{3}c_4)b_4+b_8 - \beta_{7}c_4^2, da_{11} = (t_4-b_4+\beta_{3}c_4)b_8+b_{12} - \beta_{11}c_4^3, \cdots, \\
    & da_{4n-5}=(t_4-b_4+\beta_{3}c_4)b_{4n-8} + b_{4n-4} - \beta_{4n-5}c_4^{n-1}, \\
    & da_{4n-1}=(t_4-b_4+\beta_{3}c_4)b_{4n-4} - \beta_{4n-1}c_4^{n-1},  da_{4n+3}= - \beta_{4n+3}c_4^{n-1}.
\end{align*}
Simplifying the Sullivan model by quotienting
out with contractible ideal generated by $a_3$ and $t_4$. We get the Sullivan model of $Sp(1) \setminus Sp(n+1) / Sp(1) \times Sp(n-1)$ as 

\[
(\Lambda( c_4, b_4,b_8 \dots, b_{4n-4}, a_7, \dots,a_{4n-5}, a_{4n-1}, a_{4n+3}), d),
\]
with 
\begin{align*}
    & dc_4=db_4=db_8=\cdots=db_{4n-4}=0,   \\
    & da_7 = (-b_4+\beta_{3}c_4)b_4+b_8 - \beta_{7}c_4^2, da_{11} = (-b_4+\beta_{3}c_4)b_8+b_{12} - \beta_{11}c_4^3, \cdots, \\
    & da_{4n-5}=(-b_4+\beta_{3}c_4)b_{4n-8} + b_{4n-4} - \beta_{4n-5}c_4^{n-1}, \\
    & da_{4n-1}=(-b_4+\beta_{3}c_4)b_{4n-4} - \beta_{4n-1}c_4^{n-1},  da_{4n+3}= - \beta_{4n+3}c_4^{n-1}.
\end{align*}
 By letting $t_8 = (-b_4+\beta_{3}c_4)b_4+b_8 - \beta_{7}c_4^2 $, the Sullivan model of biqoutient space $Sp(1) \setminus Sp(n+1) / Sp(1) \times Sp(n-1)$  becomes 
\[
(\Lambda( c_4, b_4, t_8, \dots, b_{4n-4}, a_7, \dots,a_{4n-5}, a_{4n-1}, a_{4n+3}), d),
\]
with 
\begin{align*}
    & dc_4=db_4=dt_8=\cdots=db_{4n-4}=0,   \\
    & da_7 = t_8 , da_{11} = (-b_4+\beta_{3}c_4)[t_8+(b_4-\beta_{3}c_4)b_4+ \beta_{7}c_4^2 ]+b_{12} - \beta_{11}c_4^3, \cdots, \\
    & da_{4n-5}=(-b_4+\beta_{3}c_4)b_{4n-8} + b_{4n-4} - \beta_{4n-5}c_4^{n-1}, \\
    & da_{4n-1}=(-b_4+\beta_{3}c_4)b_{4n-4} - \beta_{4n-1}c_4^{n-1},  da_{4n+3}= - \beta_{4n+3}c_4^{n-1}.
\end{align*}
Therefore, canceling the acyclic ideal generated by $a_7$ and $t_8$. The Sullivan model of $Sp(1) \setminus Sp(n+1) / Sp(1) \times Sp(n-1)$   reduces to 
\[
(\Lambda( c_4, b_4, b_{12}, \dots, b_{4n-4}, a_{11}, \dots,a_{4n-5}, a_{4n-1}, a_{4n+3}), d),
\]
with 
\begin{align*}
    & dc_4=db_4=db_{12}=\cdots=db_{4n-4}=0,   \\
    & da_{11} = (-b_4+\beta_{3}c_4)[(b_4-\beta_{3}c_4)b_4+ \beta_{7}c_4^2 ]+b_{12} - \beta_{11}c_4^3, \cdots, \\
    & da_{4n-5}=(-b_4+\beta_{3}c_4)b_{4n-8} + b_{4n-4} - \beta_{4n-5}c_4^{n-1}, \\
    & da_{4n-1}=(-b_4+\beta_{3}c_4)b_{4n-4} - \beta_{4n-1}c_4^{n},  da_{4n+3}= - \beta_{4n+3}c_4^{n+1}.
\end{align*}

Continuing this iterative substitution process up to the change of variable
\[
\begin{aligned}
t_{4n-4} &= (-b_4 + \beta_3 c_4) \Big[ (b_4 - \beta_3 c_4) \Big[ (b_4 - \beta_3 c_4) \cdots \Big[ (b_4 - \beta_3 c_4) \Big[ (b_4 - \beta_3 c_4)b_4 + \beta_7 c_4^2 \Big] \\
&\quad + \beta_{11} c_4^3 \Big] \cdots + \beta_{4n-13} c_4^{n-3} \Big] + \beta_{4n-9} c_4^{n-2} \Big] + b_{4n-4} - \beta_{4n-5} c_4^{n-1},
\end{aligned}
\]
we obtain a Sullivan model in which the acyclic ideal generated by \( t_{4n-4} \) and \( a_{4n-5} \) can be canceled. The resulting Sullivan model of $Sp(1) \setminus Sp(n+1) / Sp(1) \times Sp(n-1)$ simplifies to
\[
(\Lambda(c_4, b_4, a_{4n-1}, a_{4n+3}), d),
\]
with
\begin{align*}
    & dc_4=db_4=0,  da_{4n-1}=(-b_4+\beta_{3}c_4)b_{4n-4} - \beta_{4n-1}c_4^{n},  da_{4n+3}= - \beta_{4n+3}c_4^{n+1},
\end{align*}
where  \[
\begin{aligned}
b_{4n-4} &= (b_4 - \beta_3 c_4) \Big[ (b_4 - \beta_3 c_4) \Big[ (b_4 - \beta_3 c_4) \cdots \Big[ (b_4 - \beta_3 c_4) \Big[ (b_4 - \beta_3 c_4)b_4 + \beta_7 c_4^2 \Big] \\
&\quad + \beta_{11} c_4^3 \Big] \cdots + \beta_{4n-13} c_4^{n-3} \Big] + \beta_{4n-9} c_4^{n-2} \Big] +  \beta_{4n-5} c_4^{n-1}.
\end{aligned}
\]
Moreover, the cohomology algebra $H^*(Sp(1) \setminus Sp(n+1) / Sp(1) \times Sp(n-1),\mathbb{Q})$ is given by 
\[ \Lambda(c_4,b_4)\Big/ \left( (-b_4+\beta_{3}c_4)b_{4n-4} - \beta_{4n-1}c_4^{n},  - \beta_{4n+3}c_4^{n+1} \right). \]
\end{proof}

\begin{thm}
\label{thm33}   The projectivization of a quaternionic vector bundle over a $4n$-dimensional sphere $ \tau: \mathbb{H}^{n} \rightarrow E \rightarrow S^{4n}$ is given by $P(\tau) : \mathbb{H}P^{n-1} \rightarrow P(E) \rightarrow S^{4n}$. Then the total space $P(E)$ of the projectivization bundle $P(\tau)$ has a rational homotopy type of $Sp(1) \setminus Sp(2n) / Sp(2n-1)$.
\end{thm}

\begin{proof}
 The Sullivan model of $S^{4n}$ is $(\Lambda(a_{4n}, a_{8n-1}), d)$, where $da_{2} = 0$, and $da_{8n-1}=a_{4n}^{2}$. Also, the Sullivan model of $\mathbb{H}P^{n-1}$ is $\Lambda (x_{4}, x_{4n-1})$, where $dx_{4}=0$, and $dx_{4n-1} = x_{4}^{n}$. Then, the relative Sullivan model of $P(\tau)$ is
\[ (\Lambda(a_{4n}, a_{8n-1}), \mathrm{d}) \rightarrow (\Lambda(a_{4n}, a_{8n-1}) \otimes \Lambda (x_{4}, x_{4n-1}), D) \rightarrow (\Lambda (x_{4}, x_{4n-1}), d). \]
The classifying map is given by $ f : ( \Lambda(y_{8}, \cdots , y_{4n}) ,0) \rightarrow (\Lambda(a_{4n}, a_{8n-1}), \mathrm{d}).$ Also, the non-trivial Pontryagin classes are 
$ p_{n}=[f(y_{4n})]= [a_{4n}] \; \in H^{*}(S^{4n})$. In particular, the total space $P(E)$ of the projectivization bundle has a Sullivan model $$(\Lambda(a_{4n}, a_{8n-1}, x_{4}, x_{4n-1}), D),$$ where $$Da_{4n} = 0, \; Dx_{4} = 0, \; Da_{8n-1} = a_{4n}^{2}, \; Dx_{4n-1} = x_{4}^{n} + a_{4n}.$$
 Applying the change of variable $t_{4n} = x^{n}_{4} + a_{4n}$ where $Dx_{4n-1} = t_{4n}$, we get 
 \[ (\Lambda(t_{4n}, a_{8n-1}, x_{4}, x_{4n-1}), D), \]where
 \[ Da_{4n} = 0, \; Dx_{4} = 0, \; Da_{8n-1} = (t_{4n}-x_{4}^{n})^{2}, \; Dx_{4n-1} =  t_{4n}.\]
  Moreover, the Sullivan model of $P(E)$ is quasi-isomorphic to   \[ (\Lambda( x_{4}, a_{8n-1}), D), \]where
 \[  Dx_{4} = 0, \; Da_{8n-1} = x_{4}^{2n}.\]

 Now we consider the computation of the Sullivan model of $Sp(1)\setminus Sp(2n)/Sp(2n-1)$. We start with the canonical inclusions \begin{align*}
     &\varphi_H : \Lambda(v_3, v_7, v_{11}, \cdots, v_{8n-9}, v_{8n-5}, v_{8n-1}) \rightarrow \Lambda(z_4, z_8, \cdots, z_{8n-4})\\ & \text{and} \;\;  \varphi_K : \Lambda(v_3, v_7, v_{11}, \cdots, v_{8n-9}, v_{8n-5}, v_{8n-1}) \rightarrow \Lambda (b_4),
 \end{align*} 

obtained from applying Theorem \ref{thm21}. Furthermore, the Sullivan model of $Sp(1)\setminus Sp(2n)/Sp(2n-1)$ is given by 
 $$(\Lambda(b_4, z_4, z_8, \cdots, z_{8n-4}, v_3, v_7, v_{11}, \cdots, v_{8n-9}, v_{8n-5}, v_{8n-1}), d),$$
where \begin{align*}
    dv_3 =& \; z_4 - (n+1)b_4, dv_7 = z_8 - (n+1)b_4^4, dv_{11} = z_{12}-(n+1)b_4^3, \cdots, \\ dv_{8n-5}& = z_{8n-4} - (n+1)b_4^{2(n-1)}, dv_{8n-1} = -b_4^{2n}.
\end{align*} Applying the change of variable $t_4 = z_4 - (n+1)b_4$ we obtain $$(\Lambda(b_4, t_4, z_8, \cdots, z_{8n-4}, v_3, v_7, v_{11}, \cdots, v_{8n-9}, v_{8n-5}, v_{8n-1}), d),$$ 
where \begin{align*}
    dv_3 = &\;t_4, dv_7 = z_8 - (n+1)b_4^4, dv_{11} = z_{12}-(n+1)b_4^3, \cdots, \\dv_{8n-5}& = z_{8n-4} - (n+1)b_4^{2(n-1)}, dv_{8n-1} = -b_4^
    {2n}.
\end{align*}  Thus, $dv_3 = t_4$ is linear and by the cancellation of acyclic ideal generated by $v_3$ and $t_4$ we get a Sullivan model quasi-isomorphic to $$(\Lambda(b_4,  z_8, \cdots, z_{8n-4}, v_7, v_{11}, \cdots, v_{8n-5}, v_{8n-1}), d),$$
where \begin{align*} dv_7 = &~z_8 - (n+1)b_4^4, dv_{11} = z_{12}-(n+1)b_4^3, \cdots, \\
dv_{8n-5} &= z_{8n-4} - (n+1)b_4^{2(n-1)},\; dv_{8n-1} = -b_4^{2n}.\end{align*} 
Also, applying the substitution $t_8 =  a_8 - (n+1)b_4^4$ and canceling the acyclic ideal generated by $v_7$ and $t_8$ we obtain the Sullivan model which is quasi-isomorphic to  $$(\Lambda(b_4, z_{12} \cdots, z_{8n-4}, v_{11}, \cdots,  v_{8n-5}, v_{8n-1}), d),$$
where \[dv_{11} = z_{12}-(n+1)b_4^3, \cdots, dv_{8n-5} = z_{8n-4} - (n+1)b_4^{2(n-1)}, dv_{4n-1} = -b_4^{2n}.\] By continuing with this process of changing variables and canceling acyclic ideals we obtain a Sullivan model of  $Sp(1)\setminus Sp(2n)/Sp(2n-1)$ as  
$$(\Lambda(b_4, v_{8n-1}), d),$$ where $db_4 = 0$ and $dv_{8n-1} = -b_4^{2n}.$ Moreover, we consider a morphism 
\[ \eta : (\Lambda(b_4, v_{8n-1}), d) \to (\Lambda(x_4, a_{8n-1}), d), \]
 where $dx_4 = 0$ and $da_{8n-1} = x_4^{2n}.$ Therefore, the morphism $\eta$ defined  by $\eta(b_4)=x_4$ and $\eta(v_{8n-1})=-b_{8n-1}$ is a quasi-isomorphism. Hence, $P(E)$ has a rational homotopy type of a biquotient space $Sp(1)\setminus Sp(n)/Sp(n-1).$ In particular, we write $$P(E)\simeq_{\mathbb{Q}} Sp(1)\setminus Sp(2n)/Sp(2n-1).$$

\end{proof}

\begin{thm}
\label{thm34}    The total space $P(E)$ of the projectivization bundle  $P(\tau) : \mathbb{H}P^{1} \rightarrow P(E) \rightarrow \mathbb{H}P^{2}$ has a rational homotopy type that is quasi-isomorphic to biquotient space $Sp(1)\setminus Sp(3) / Sp(1) \times Sp(1)$.
\end{thm}

\begin{proof} A cdga model of $\mathbb{H}P^{1}$ is $(\Lambda (x_{4},\; x_{7}), d)$ with $dx_4=0$ and $dx_{7} = x_{4}^{2}$. The cdga model of $\mathbb{H}P^{2}$ is $(\Lambda (y_{4}, \; y_{11}),d)$, where $dy_4=0$ and $dy_{11} = y_{4}^{3}$. Then with reference to \textit{Proposition \ref{prop22},} the cdga model of the total space $P(E)$ of the projectivization bundle $$P(\tau) : \mathbb{H}P^{1} \rightarrow P(E) \rightarrow \mathbb{H}P^{2}$$ is given by 
$$(\Lambda(x_{4},\; y_{4},\; x_{7},\; y_{11}), \; d), $$
 with $dx_{4}=dy_{4}=0$,
 $ dx_{7}=x_{4}^{2}+ x_{4}y_{4} + y_{4}^{2}$ and $dy_{11} = y_{4}^{3}$. Moreover, the cohomology algebra of $P(E)$ is $$ H^{*}(P(E), \mathbb{Q}) = \Lambda (y_4, x_{4}) /\left(x_{4}^{2}+x_{4}y_{4}+y_{4}^{2}, \; y_4^3 \right). $$

Now we compute the Sullivan model of the biquotient space $Sp(1)\setminus Sp(3) / Sp(1) \times Sp(1)$. The inclusions morphisms are $i_{H} : H = Sp(1) \times Sp(1) \hookrightarrow Sp(3)$ and $i_{K} : K = Sp(1) \hookrightarrow Sp(3) = G.$
Also, the induced models of $B_{i_{H}} : BH \rightarrow BG$ and $B_{i_{K}} : BK \rightarrow BG$ are $\varphi_{H} : (\Lambda(v_{4}, v_{8}, v_{12}), 0) \rightarrow (\Lambda(a_{4}, c_{4}), 0)$ and $\varphi_{K} : (\Lambda(v_{4}, v_{8}, v_{12}), 0) \rightarrow (\Lambda(b_{4}), 0)$ respectively.
 The biquotient $K\setminus G / H$ is the homotopy pullback of the following diagram
\begin{center}
\begin{tikzpicture}[node distance=1.5cm and 2cm]
\node(a){$BH$};
\node(b)[left = of a]{$K\setminus G / H$};
\node(e)[below = of a]{$BG$};
\node(f)[below = of b]{$BK$};
\draw [<-](a)-- node[above = 0.1cm]{}(b);
\draw [<-](e)-- node[above = 0.1cm]{$B_{i_{K}}$}(f);
\draw [<-](e)-- node[right = 0.1cm]{$B_{i_{H}}$}(a);
\draw [<-](f)-- node[above = 0.1cm]{}(b);
\end{tikzpicture}
\end{center}
Applying \textit{Theorem \ref{thm21},} the Sullivan model of  $K\setminus G / H = Sp(1)\setminus Sp(3) / Sp(1) \times Sp(1)$ is 
\[((\Lambda(a_{4}, c_{4}) \oplus \Lambda b_{4}) \otimes \Lambda(v_{3}, v_{7}, v_{11}), d),  \] where 
  $dv_{3}=a_{4}+c_{4}-3b_{4},
dv_{7}=a_{4}c_{4}-3b^2_{4}, 
dv_{11}=-b^3_{4}. $ Using the substitution $t_{4} = a_{4}+c_{4}-3b_{4}$ the Sullivan model $((\Lambda(a_{4}, b_{4},  c_{4}) \otimes \Lambda(v_{3}, v_{7}, v_{11}), d)$  is quasi-isomorphic to
$$((\Lambda(a_{4}, b_{4},  t_{4}) \otimes \Lambda(v_{3}, v_{7}, v_{11}), d),$$  where $ dv_{3}=t_{4},
dv_{7}=a_{4}(t_4+3b_{4}-a_{4})-3b^2_{4}$ and $ 
dv_{11}=-b_{4}^3.$ By the cancellation of the acyclic ideal generated by $v_3$ and $t_4$ the Sullivan model $(\Lambda(a_{4}, b_{4},  t_{4}, v_{3}, v_{7}, v_{11}), d)$ is quasi-isomorphic to \[(\Lambda(a_{4}, b_{4}, v_{7}, v_{11}), d),\] where $da_{4}=db_{4}=0, dv_{7}=3b_{4}(a_{4}-b_{4})-a^2_{4}$ and $dv_{11}=-b_{4}^3.$ Then applying the change of variables using $\overline{x}_4=a_4-b_4$, $\overline{y}_4=-b_4$, $\overline{x}_7=v_7$ and $\overline{y}_{11}=v_{11}$ we get \[(\Lambda(\overline{x}_{4}, \overline{y}_{4}, \overline{x}_{7}, \overline{y}_{11}), d),\] where $d\overline{x}_{4}=d\overline{y}_{4}=0, d\overline{x}_{7}=-\overline{x}_{4}^2 - \overline{x}_{4}\overline{y}_{4} - \overline{y}^2_{4}$ and $d\overline{y}_{11}=-\overline{y}_{4}^3.$ Solving the system of equations given by $  3b_{4}(a_{4}-b_{4})-a^2_{4} = -\overline{x}_{4}^2 - \overline{x}_{4}\overline{y}_{4} - \overline{y}^2_{4}$, $-b_4^3 = -\overline{y}_{4}^3$ 
we obtain $\overline{x}_4=a_4-b_4$ and $\overline{y}_4=-b_4$.

A morphism $$\Bar{f} : P(E) \rightarrow Sp(1)\setminus Sp(3) / Sp(1) \times Sp(1)$$ yields the following cdga model $$ f :  (\Lambda(\overline{x}_{4}, \overline{y}_{4}, \overline{x}_{7}, \overline{y}_{11}), d) \rightarrow (\Lambda(x_{4}, y_{4}, x_{7}, y_{11}), d),$$ such that 
$ f(\overline{x}_{4}) = x_{4}, \; f(\overline{x}_{7}) = x_{7},$ $f(\overline{y}_{4}) = y_{4}, \;
f(\overline{y}_{11}) = y_{11} 
$. Hence, $\Bar{f}$ is a quasi-isomorphism and the total space $P(E)$ of the projectivization bundle $P(\tau)$ has the rational homotopy type of $Sp(1)\setminus Sp(3) / Sp(1) \times Sp(1)$.

\end{proof}


\end{document}